\documentclass[draft, reqno]{amsart}
\usepackage[all]{xy}                        %

\CompileMatrices                            

\UseTips                                    

\input xypic
\usepackage[bookmarks=true]{hyperref}       

\usepackage{amssymb,latexsym,amsmath,amscd}
\usepackage{amssymb}
\usepackage{xspace}
\usepackage{enumerate}
\usepackage{graphicx}
\usepackage{vmargin}
\usepackage{todonotes}
\usepackage{xcolor}

\reversemarginpar

\theoremstyle{plain}
\newtheorem{theorem}{Theorem}[section]
\newtheorem*{theorem*}{Theorem}
\newtheorem{proposition}[theorem]{Proposition}
\newtheorem{corollary}[theorem]{Corollary}
\newtheorem{lemma}[theorem]{Lemma}

\theoremstyle{definition}
\newtheorem{definition}[theorem]{Definition}

\newtheorem{notation}[theorem]{Notation}

\newtheorem{remark}[theorem]{Remark}
\newtheorem{example}[theorem]{Example}

\newcommand{\enm}[1]{\ensuremath{#1}}          %

\newcommand{\cal}[1]{\mathcal{#1}}

\newcommand{\PP}{\enm{\mathbb{P}}}

\newcommand{\KK}{\enm{\mathbb{K}}}

\newcommand{\Gg}{\enm{\cal{G}}}

\newcommand{\Ii}{\enm{\cal{I}}}

\newcommand{\Ll}{\enm{\cal{L}}}
\newcommand{\Mm}{\enm{\cal{M}}}

\newcommand{\Oo}{\enm{\cal{O}}}

\newcommand{\Rr}{\enm{\cal{R}}}
\newcommand{\Ss}{\enm{\cal{S}}}

\newcommand{\spa}[1]{\langle {#1}\rangle}

\begin{document}

\title[Contact loci and identifiability]{On the dimension of contact loci and the identifiability of tensors}

\author[]{Edoardo Ballico, Alessandra Bernardi and Luca Chiantini}
\address[Edoardo Ballico]{Dipartimento di Matematica,  Univ. Trento, Italy}
\email{edoardo.ballico@unitn.it }
\address[Alessandra Bernardi]{Dipartimento di Matematica,  Univ. Trento, Italy}
\email{alessandra.bernardi@unitn.it }
\address[Luca Chiantini]{Dipartimento di Ingegneria dell'Informazione e Scienze Matematiche,  Univ. Siena, Italy}
\email{luca.chiantini@unisi.it}

\begin{abstract}
Let $X\subset \PP^r$ be an integral and non-degenerate variety. Set $n:= \dim (X)$. We prove that  if the $(k+n-1)$-secant
variety of $X$ has (the expected) dimension $(k+n-1)(n+1)-1<r$ and $X$ is not uniruled by lines, then $X$ is not $k$-weakly defective 
and hence the $k$-secant variety satisfies identifiability, i.e. a general element of it
is in the linear span of a unique $S\subset X$ with $\sharp (S) =k$. We apply this result to many Segre-Veronese varieties 
and to the identifiability of Gaussian mixtures $\Gg _{1,d}$. If $X$ is the Segre embedding of a multiprojective space we prove 
identifiability for the $k$-secant variety (assuming that the $(k+n-1)$-secant variety has dimension $(k+n-1)(n+1)-1<r$, 
this is a known result  in many cases), beating several bounds on the identifiability of tensors.\end{abstract}

\maketitle

\section{Introduction}
The computation of the decompositions of tensors and the related evaluation of their ranks recently received great 
interest from people working in algebraic geometry. Indeed, a geometric analysis on projective spaces of tensors, 
Segre embeddings and their secant varieties produced currently several new results in the field.

From the geometric point of view, the study of decompositions and ranks can be turned to a general problem. For any (irreducible)
algebraic variety $X$ which spans a projective  space $\PP^r$, one can define the {\it  $X$-rank} $r_X(q)$ of
a point  $q\in \PP^r$ as the minimal cardinality of a finite set 
$S\subset X$ such that $q\in \spa{ S}$, where $\spa{ \,\, }$ denotes the linear span. Any set
$S$ for which the minimum is attained is a {\it minimal $X$-decomposition} of $q$. 

In tensor analysis, one interprets $\PP^r$ as a space of tensors of given type $(n_1+1)\times\dots\times (n_s+1)$
(mod scalars), and $X$ as the locus of tensors of rank $1$. In geometric terms, $X$ is the multiprojective 
product $\PP^{n_1}\times\dots\times \PP^{n_s}$, naturally embedded in $\PP^r$ via the Segre map. Then, the notions 
of $X$-rank and minimal decompositions match with the usual multilinear algebra  notion. The set of tensors
of fixed rank $k$ corresponds to an open subset of the $k$-th secant variety $\sigma_k(X)$ of $X$. Properties
of the rank decompositions can be  derived from geometric properties of $X$, which determine some features
of $\sigma_k(X)$. Examples of such analysis are contained in several papers, e.g. see
\cite{a1, ah,  cc1, cc2, gm, m1}, and the book \cite{l}.

 In this paper we study the problem of the {\it identifiability} of generic points $q\in\PP^r$ of given $X$-rank $k$. 
 A point $q$ is $X$-identifiable when the set  $\Ss (q, X)$ of minimal $X$-decompositions of $q$ is a singleton; 
 in any case the cardinality of  $\Ss (q, X)$ is called the \emph{secant degree of $q$} or the $k$-th secant degree of $q$.
 In tensor analysis the identifiability of a tensor is fundamental for algorithms of data compression and
 for the analysis of mixture models (see \cite{ddl1,ddl2,ddl3,sb}). We will say that $X$ is {\it $k$-identifiable} 
 if a generic point of $\sigma_k(X)$ is $X$-identifiable.
 
 We first work over an algebraically closed field of characteristic $0$.
  Our investigation is based on the notion of weak defectivity introduced by the last author and C. Ciliberto in \cite{cc1}.
 A projective variety $X$ is {\it weakly $k$-defective} if hyperplanes of $\PP^r$, which are tangent to $X$ at
 $k$ general points, are indeed tangent along a positive dimensional subvariety. It turns out (see \cite[Theorem 2.5]{cc2})
 that when $r>k(n+1)-1$ then  $X$ is $k$-identifiable, unless it is weakly $k$-defective (unfortunately the reverse is not true).
 Therefore the knowledge on the {non-$k$} weak defectivity provides an effective way to determine geometrically the identifiability of generic points.
 This approach has been applied to tensors and linear systems of tensors  in a series of papers (\cite {a, bbcc, bv, bc,  co, cov,
 cov1}), where identifiability problems are partially or totally solved. 

A challenging task is therefore to get criterions that can establish identifiability (or not) in the cases when weak defectivity arises. 
 To this purpose there is a related notion that is very helpful, meaning the \emph{tangential weak defectivity}. 
 Our new analysis starts with the following observation (see \cite{cc3}). If $X$ is a weakly $k$-defective variety of dimension $n$, then a general choice
 of a set $A$ of $k$  points on $X$ determines a positive dimensional {\it tangential $k$-contact locus $\Gamma_k(A)$}: 
 the tangency locus of the span of the tangent spaces to $X$ at the points of $A$. If for a general choice of $A$  we have that $\dim \Gamma_k(A)>0$ 
 than $X$ is  \emph{tangential $k$-weakly defective}. It turns out that the $k$-th secant degree of a general element of $X$ is {related
 with the $k$-th  secant degree of the contact locus $\Gamma_k(A)$ for a general choice of $A$ (cf. Lemma \ref{gamma}, Corollary \ref{corollary:sec:deg}, 
 \cite[Theorem 2.4]{cc2} and part (iii) of \cite[Proposition 3.9]{cc3}). We will exploit in more details the relation in Section \ref{secdeg}.}
 
 Moreover,  when $r\geq (k+1)(n+1)-1$, then the weak $k$-defectivity 
 of $X$ implies the weak $(k+1)$-defectivity. Moreover, we prove that, under some geometric assumption on $X$,
if we take a generic subset $A$ as above and let $A'$ be obtained by adding to $A$ a generic point of $X$, then the dimension
of the tangential contact locus $\Gamma_{k+1}(A')$ is strictly bigger than the dimension of $\Gamma_k(A)$. As the dimension
of contact loci is bounded by $n-1$, by repeating the previous construction we obtain the following result.

\begin{theorem*} (see Theorem \ref{ip1}).
Let $X\subset \PP^r$ be an integral and non-degenerate $n$-dimensional variety and let $k$ be a positive integer. 
Assume that $X$ is not uniruled by lines and that $\dim \sigma _{k+n-1}(X) ={(k+n-1)}(n+1)-1 <r$.
Then $X$ is not weakly $j$-defective for every $j\leq k$.
\end{theorem*}

It follows from the previous analysis that, with the assumptions of Theorem \ref{ip1},  $X$ is $j$-identifiable for
every $j\le k$. 

The result can be directly applied e.g. to Segre-Veronese embeddings of products of projective spaces, provided 
that the multidegree $(d_1,\dots ,d_s)$ satisfies  $d_i>1$ for all $i$.
When some degree $d_i$ is $1$, as in the case of Segre embeddings (where $d_i=1$ for all $i$), then Theorem \ref{ip1} 
does not apply directly. However, for the Segre product we can modify the argument of Theorem \ref{ip1},
and prove that the same conclusion holds when for each $i\in \{1,\dots ,s\}$ 
there is  $j\in \{1,\dots ,s\}\setminus \{i\}$ such that $n_j=n_i$ (see Theorem \ref{ip3}).

Combining  Theorem \ref{ip3} with the main result in \cite{cgg}, for the Segre embedding $X$ of $s\geq 5$ 
copies of $\PP^1$, we find that $X$ is not weakly $k$-defective, hence it is $k$-identifiable, 
whenever $(k+s-1)(s+1) < 2^s-1$  (see Corollary \ref{ip4}).

Corollary \ref{ip4} improves our previous knowledge on the weak defectivity of Segre embeddings
of many copies of $\PP^1$, hence on the identifiability of generic {\it binary} tensors. 
Indeed it extends the range where we know that $k$-identifiability holds 
up to $k\leq \frac{s^2-1}{s+1} -s$, which beats all previously known bounds (see Section 4 of \cite {bco}) as $s$ grows.
Since identifiability cannot hold, for trivial dimensional reasons, when  $k> s^2/(s+1)-1$, then we are quite close
to a complete  analysis of the generic identifiability of binary tensors. We also get almost optimal results
for the Segre embedding of $(\PP^m)^s$, $m>1$, see Corollary \ref{e1}) (we use the very strong \cite[Theorem 3.1]{ah0}).

See Example \ref{r+1} for the case of Gaussian mixtures $\Gg_{1,d}$. 
\smallskip

When the assumptions of Theorem \ref{ip1} does not hold, one cannot conclude the non weak $k$-defectivity
of $X$. Yet, even if $X$ is weakly $k$-defective, it can still be $k$-identifiable. 
We prove the following result.

{\begin{theorem*} (see Theorem \ref{ipx})
Let $X$ be the Segre embedding of a product $\mathbb{P}^{n_1}\times \cdots \times \mathbb{P}^{n_q}$ 
in $\mathbb{P}^r$, with $n=\sum n_i=\dim(X)$. Let $k$ be a positive integer. 
Assume that  {$\dim \sigma _{k+n-1}(X) =(k+n-1)(n+1)-1< r$} (in particular
 $X$ is not $(k+n-1)$-defective).
Then $X$ is $j$-identifiable for every $j\le k$.
\end{theorem*}}

It would be interesting to overcame the assumption ``$X$ not uniruled by lines '' of Theorem \ref{ip3} in other cases, 
apart the Segre varieties covered by Theorem \ref{ipx}, because it would expand the known cases of identifiability 
(e.g. for Grassmannians see \cite{aop1}, \cite{bdg}, \cite{b}, \cite{cgg1}; for the tangential variety of the Veronese 
varieties see \cite{av}). We just recall that conjecturally only four Grassmannians have some defective secant variety.

We stress that, in particular, Remark \ref{comput} suggests that if one is interested in the $k$-identifiability
of some specific Segre product then the computational procedure induced by Theorem \ref{ipx} can be faster than
previously known direct procedures.
 \medskip

In the last section of the paper we work over an algebraically closed base field $\KK$ of arbitrary characteristic. 
To include varieties $X$ defined in positive characteristic, we will assume that $X$ is \emph{not very strange}
(see Definition \ref{defvs}). Under this mild assumption, we prove the following.

\begin{theorem}\label{ii1}
Let $M$ be an integral projective variety and $\Ll$, $\Rr$ line bundles on $M$ with $\Ll$ very ample. 
Set $n:= \dim M$. Fix a linear subspace $V\subseteq H^0(\Ll)$ such that the associated
map $j: M\to \PP^r$ is an embedding with image not very strange. Fix an integer $k \leq \min \{h^0(\Rr), \dim V -n-2\}$. 
Let $W\subseteq H^0(\Ll\otimes \Rr)$ be the
image of the multiplication map $V\otimes H^0(\Rr )\to H^0(\Ll \otimes \Rr )$ and let 
$X\subset \PP^r$, $r =\dim W -1$, be the embedding induced by $W$.
Then $\sigma _k(X)$ has dimension $kn+k-1$ and $\sharp (\Ss (q,X)) =1$ for a general $q\in \sigma _k(X)$.
\end{theorem}

We will apply Theorem \ref{ii1} to the Segre-Veronese embeddings of a multiprojective space in arbitrary characteristic
(see Example \ref{ii2}), where we do not have an analogue of \cite[Theorem 2.5]{cc2}, hence the analysis
of the previous section does not apply.
Notice that  Theorem \ref{ii1} does not apply to Segre embeddings, because it requires
$\Ll$ very ample (or at least birationally very ample) and $h^0(X,\Rr )>1$.
\smallskip

{The authors wish to thank C. Bocci, for his help to understand the case of almost unbalanced tensors
(see Example \ref{aumb} below) and the anonimous referee for pointing out an improvement in the formula  
after Corollary \ref{ip4}.
}

\section{General notation}

In this section we work with projective spaces $\PP^r$ defined over an algebraically closed field $\KK$.

Let $X\subset \PP^r$ be an integral and non-degenerate $n$-dimensional variety.  We will denote
with $X_{\mathrm{reg}}$ the set of regular points in $X$. For any $x\in X_{\mathrm{reg}}$ we will denote
with $T_xX$ the tangent space to $X$ at $x$.

For any  integer $k\ge 2$ we will denote with $\sigma _k(X)$ the $k$-th secant variety of $X$ which is an irreducible projective variety.
 If $k(n+1)-1 \leq r$, then $k(n+1)-1$ is the expected dimension of $\sigma_k(X)$ and $X$ is {\it $k$-defective}
 if $\dim \sigma _k(X) <k(n+1)-1$.
 
 We will often use the fact that if $X$ is not $k$-defective, then, for all $k'\leq k$, $X$ is not $k'$-defective. 

\begin{definition}
For $q\in \PP^r$ general, the $X$-rank $r_X(q)$ of
a point  $q\in \PP^r$ is the minimal cardinality of a finite set 
$A\subset X$ such that $q\in \spa{ A}$. 

We define $\Ss_X(q)$ as the collection
of subsets $A\subset X$ of cardinality $r_X(q)$ such that $q\in \spa{ A}$.
When $\Ss_X(q)$ is finite, we will write $\alpha_X(q)$ for the cardinality of  $\Ss_X(q)$.
\end{definition}

\begin{remark}
Since $X$ is irreducible, then the secant varieties $\sigma_k(X)$ are irreducible.
Moreover a general point of $\sigma_k(X)$ has $X$-rank $k$.  

By definition, the set $\Ss_X(q)$ is finite for $q\in\sigma_k(X)$ general if and only if 
$\dim \sigma_k(X)=k(n+1)-1 \leq r$. In this case, we will write simply $\alpha_k$ 
for the cardinality of  $\Ss_X(q)$, where $q\in\sigma_k(X)$ is generic. Such an $\alpha_k$ is often called the \emph{$k$-th secant degree of $X$}.
\end{remark}

The following result holds (with the same elementary proof) in arbitrary characteristic.

\begin{proposition}\label{propa1}
Assume $\dim \sigma _{k+1}(X) =(k+1)(n+1) -1 $ and $\alpha _k>1$. Fix a general $q\in \sigma _{k+1}(X)$, 
any $A\in \Ss (q)$ and any $a\in A$. There are at least $\alpha _k$ elements $B_1 \ldots, B_{\alpha_{k}}\in \Ss (q)$ (one of them being $A$) such
that $a\in B_i$ for all $i=1, \ldots , \alpha_k$.
\end{proposition}
\begin{proof}
Since $\dim \sigma _{k+1}(X) =(k+1)(n+1) -1$, we have $\dim \sigma _k(X) =k(n+1) -1$ 
and hence $\alpha _k$ is a well-defined finite positive integer. Fix a general $q\in \sigma _{k+1}(X)$ 
and take any $A\in \Ss (q)$ and any $a\in A$.  Set $E:= A\setminus \{a\}$. Since $q$ is general 
in $\sigma _{k+1}(X)$ and $\dim \sigma _{k+1}(X) =(k+1)(n+1) -1$,
a dimensional count gives that for each $D\in \Ss (q)$ the set $D$ is a general element of the 
$(k+1)$-symmetric product of $X$.
Thus $E$ is a general subset of $X$ with cardinality $k$. Since $\sharp (A) =r_X(q)$, we have 
$q\notin \spa{ A'}$ for any $A'\subsetneq A$; in particular $q\notin E$. Thus $\spa{ \{a,q\}}$ is a line, which meets 
$\spa{E}$ in a unique point $q'$ (because $\langle E \rangle + \langle a,q \rangle= \langle A \rangle$) and $q'\notin \spa{E'}$ for any $E'\subsetneq E$. 
For a general $q\in \spa{A}$, $q'$ is a general element of $\spa{E}$. Since $E$ is general, 
$q'$ is general in $\sigma _k(X)$. Thus $r_X(q')=k$, $E\in \Ss (q')$ and $\sharp (\Ss (q')) = \alpha _k>1$.
Fix any $F\in \Ss (q')$. We have $q\in \spa{ \{F,a\}}$. Since $r_X(q) =k+1$, 
we have $a\notin F$ and $F\cup \{a\}\in \Ss (q)$. Now we have exactly $\alpha_k$ such $F$'s, say $F_1, \ldots , F_{\alpha_{k}}$, 
the $B_i$ of the statement are $B_i=\{F_i,a\}$ for $i=1, \ldots , \alpha_k$ concluding the proof.
\end{proof}

\section{Weak defectivity and identifiability}\label{charo}

In this section, {as well as in the next two sections,} we assume that the algebraically closed base field $\KK$ has characteristic $0$.
\smallskip

For a subset $A$ of cardinality $k$ contained in the regular locus $X_{\mathrm{reg}}$ we set:
$$M_A:= \spa{ \cup _{x\in A} T_xX}.$$ 
By Terracini's lemma (\cite[Corollary 1.11]{a1}), for a general choice of $A$ the space $M_A$
is the tangent space to $\sigma_k(X)$ at a general point $u\in\spa{ A}$. Thus $\dim M_A = \min \{ k(n+1)-1, \dim \langle X \rangle\}$
for $A$ general, unless $X$ is $k$ defective.

\begin{definition}\label{notazione}
The \emph{tangential $k$-contact locus} $\Gamma _k(A)$  is the closure in $X$ of the union of all the irreducible components which contain 
at least one point of $A$, of the locus of points of $X_{\mathrm{reg}}$ where $M_A$ is tangent to $X$.
\\
We will write $t_k:=\dim \Gamma _k(A)$ for a general choice of $A$.
\\
We say that $X$ is \emph{weakly $k$-defective} if $t_k >0$.
\\
We will write just $\Gamma _k$ instead of $\Gamma _k(A)$, when $A$ is general.
\end{definition}

With the terminology just introduced we have the following remark:
\begin{remark}\label{rem:catena:weak}\cite[Remark 3.1(iii) at p. 159]{cc1} If $\sigma_{k+1}(X) \subsetneq \mathbb{P}^r$ 
has the expected dimension $(k+1)\dim X+(k+1)-1<r$ and $X$ is weakly $k$-defective, then $X$ is also weakly $(k+1)$-defective.
\end{remark}

We will need the following result of \cite{cc3}.

\begin{proposition}
For a general choice of $A\subset X_{\mathrm{reg}}$, the algebraic set $\Gamma _k:=\Gamma _k(A)$ is equidimensional 
and either it is irreducible or it has exactly $k$ irreducible components,  each of them containing a different point of $A$.
\end{proposition}
\begin{proof} It is Proposition 3.9 of \cite{cc3}. (Notice that our $\sigma_k(X)$ is called the
{\it $(k-1)$-th secant variety} in \cite{cc3}, so that all the $k$'s  there should be here replaced by $k+1$, in order
to apply the results of \cite{cc3} to our terminology). 
\end{proof}

\begin{definition}\label{types}
Following \cite{cc2} and \cite{cc3}, we say that $\Gamma_k$ is of the \emph{type I} (resp. \emph{type II}) 
if $\Gamma _k$ is irreducible (resp. it is not irreducible) for a general choice of $A$. 
\end{definition}

The fundamental technical results of the present paper is the following lemma.
It shows how one can control the growth of the dimension of contact loci.
Mainly, one must take care of the case of contact loci of type II, which are not irreducible.

\begin{lemma}\label{a2} Let $X\subset \mathbb{P}^r$ be an irreducible projective variety such that  
$\dim \sigma _{k+1}(X) =(k+1)(n+1) -1 <r$ and assume that $X$ is weakly $k$-defective. Then $X$ is weakly $(k+1)$-defective.
Let $B$ (resp. $A$) be a general subset of $k$ (resp. $k+1$) points of $X$. Call $\Gamma _k$ (resp. $\Gamma _{k+1}$) the tangential 
contact locus of $X$ at the points of $B$ (resp. $A$). 
Then either:
 
 \begin{itemize}
 \item $\dim \Gamma _k<\dim \Gamma _{k+1}$; or 
 
 \item both $\Gamma _k$ and $\Gamma _{k+1}$ have type II, each of the irreducible components
of $\Gamma _k$ and $\Gamma _{k+1}$ is a linear space and these components are linearly independent, i.e.
$$\dim \spa{ \Gamma _{k+1}} =\dim \spa{\Gamma _k} +\dim \Gamma _k+1 = (k+1)(\dim \Gamma _k +1)-1. $$
\end{itemize}
\end{lemma}

\begin{proof}
Since $\dim \sigma _{k+1}(X) =(k+1)(n+1) -1$, we have $\dim \sigma _k(X) =k(n+1) -1$. 
By Remark \ref{rem:catena:weak}, $X$ is weakly $(k+1)$-defective. By \cite[first two lines after Definition 1.2]{cc1}, 
$X$ is not $i$-defective for $i=1,\dots,k+1$.

By \cite[Proposition 3.9]{cc3} we know that both $\Gamma _k$ and 
$\Gamma _{k+1}$ are equidimensional, that $A$ is in the smooth locus of $\Gamma _{k+1}$, that $B$ is in the smooth locus
of $\Gamma _k$ and that either $\Gamma _k$ (resp. $\Gamma _{k+1}$) is irreducible or it has exactly $k$ (resp. $k+1$) 
irreducible components,  each of them containing exactly one point of $B$ (resp. $A$). 

Since $B$ is general, we may assume that $B\subset A$. Thus $\Gamma _k\subseteq \Gamma_{k+1}$.
Since $\Gamma_k$ is a proper subvariety of $X$ and $\Gamma _{k+1}$ contains $B$ plus a general point of $X$, then
 $\Gamma _k\neq \Gamma _{k+1}$. Thus if either $\Gamma _k$ or $\Gamma _{k+1}$ is irreducible, then
 $t_k<t_{k+1}$ (where $t_k$ is defined in Definition \ref{notazione}).
 
Now assume that both $\Gamma _k$ and $\Gamma _{k+1}$ are reducible, hence of type II, and that $t_k=t_{k+1}=:t$. 
Let $B=\{P_1,\dots,P_k\}$
and $A=\{P_{k+1}\}\cup B$. Call $\Gamma^i$ the component of $\Gamma_{k+1}$ passing through $P_i$, so that also
$\Gamma_k=\Gamma^1\cup\dots\cup\Gamma^k$. Call ${\Pi^i}$ the linear span of $\Gamma^i$. Of course, by the generality of $A$, 
all the ${\Pi^i}$'s have the same dimension. Assume that $\dim {\Pi^i}>\dim \Gamma^i$ and let us show that we get a contradiction. 

For $i=2,\dots,k+1$ call $L_i$ the span of ${\Pi^1\cup\dots\cup\Pi^i}$, which is also the span of $\Gamma^1\cup\dots\cup\Gamma^i$.
By  \cite[Proposition 2.5]{cc2} we know that, for all $i$, $\dim(L_i)\geq it+ i-1$, since $X$ is not $i$-defective.
Moreover by \cite[Proposition 3.9]{cc3} we know that $\dim(L_k)=kt+k-1$ and $\dim(L_{k+1})=(k+1)t+k$.
If $\dim(L_{k-1})>(k-1)t+k-2$, then by Grassmann formula, the intersection of $L_{k-1}$ and $\Pi_k$ has dimension
greater than $(k-1)t+k-2+\dim(\Pi_k)-kt-k+1=\dim(\Pi_k)-t-1$, so that, by Grassmann formula again, 
$$\dim(L_{k+1})< \dim(L_k)+\dim(\Pi_{k+1})-\dim(\Pi_i)+t+1=(k+1)t+k,$$
a contradiction. Thus $\dim(L_{k-1})=(k-1)t+k-2$. Arguing by induction, we obtain that $\dim(L_i)=it+i-1$
for $i=2,\dots,k+1$. In particular $\dim({\Pi^1\cap\Pi^2})=2\dim({\Pi^1})-2t-1\geq 0$. By the generality of $A$ again,
$\dim({\Pi^i\cap\Pi^j})=\dim({\Pi^1\cap\Pi^2})$ for all $i,j$. Thus
${\Pi^3}$ intersects $L_2$ in dimension at least  $\dim{\Pi^1\cap\Pi^2})$. It follows that $L_3$ has dimension 
$$\dim(L_3)\leq \dim(L_2)+\dim({\Pi^3})-2\dim({\Pi^1})+2t+1=(2t+1)-\dim({\Pi^i})+2t+1<3t+2,$$
which yields the claimed contradiction.
\end{proof}

{
\section{Remarks on the secant degree of contact loci}\label{secdeg}

We collect in this sections some remarks on the general {\it tangential}  contact loci $\Gamma_k$ of $X$, which clarify the relations
between the $k$-th secant degree of $X$ and the $k$-th secant degree of a general $\Gamma_k$.
To do that, we need to distinguish the case where $\Gamma_k$ is of type I (i.e. irreducible) from the case where
$\Gamma_k$ is of type II.
\smallskip

\begin{notation}
{\it For the rest of the section, let $X\subset \PP^r$ be an integral and non-degenerate $n$-dimensional variety 
and let $k$ be a positive integer.  Assume that $X$ is weakly $k$-defective with
$$\dim \sigma _{k+n-1}(X) ={(k+n-1)}(n+1)-1< r. $$ 
Call $\Gamma_k$ the tangential $k$-contact locus of $X$ at a general set $A=\{P_1,\dots,P_k\}$ of $k$ points of $X$,
and call ${\Pi_k}$ the linear span of $\Gamma_k$. }
\end{notation}

The following technical result is the analogue of \cite[Theorem 2.4]{cc2} for the tangential contact locus $\Gamma_k$.
It is implicitly contained in part (iii) of \cite[Proposition 3.9]{cc3}, but we prefer to make it explicit here.

\begin{lemma}\label{gamma} For a general $Q\in\spa{A}\subset{\Pi_k}$, all subsets $B\subset X$ of cardinality $k$ such that $Q\in\spa{B}$ are
contained in the span of $\Gamma_k$. 
\end{lemma}
\begin{proof} Notice that, by the generality of $A$, $Q$ is also a general point in the secant variety $\sigma_k(X)$.
Since $Q$ is general, the tangent space to $\sigma_k(X)$ at $Q$ is also the span of the tangent spaces to $X$ at the points of $B$,
by the Terracini's Lemma. The claim now follows from \cite[proposition 3.9 (iii)]{cc3}.
\end{proof}

\begin{corollary}\label{corollary:sec:deg} 
The $k$-secant degree of a projective variety $X$ is equal to the $k$-secant degree of the intersection of $X$ with the
span of  its tangential $k$-contact locus.
\end{corollary}

When the general contact locus of $X$ is of type I, we can give a more explicit statement.

\begin{proposition} Assume that $\Gamma_k$ is irreducible (i.e. type I).
Assume the existence of a component $W$ of ${\Pi_k}\cap X$, 
different from $\Gamma_k$, such that ${\Pi_k}$ is tangent to $X$ along $W$. Then for all $a,b$ with $a>0$ and $a+b=k$,
the join of the $a$-th secant variety of $W$ with the $b$-th secant variety of $\Gamma_k$ cannot cover ${\Pi_k}$.
\end{proposition}
\begin{proof} Assume on the contrary that the join covers ${\Pi_k}$. Since $\sigma_k(X)$ has the expected dimension and it is
equal to the union of the spaces ${\Pi_k}$'s, when the $k$-th points $P_1,\dots, P_k$ vary on $X$, then for dimensional reasons
there cannot be a subvariety $Y\subsetneq X$ which contains all the respective components $W$'s. Thus, for $P_1,\dots, P_k$ general,
 a general point $Q$ of $W$ is a general point of $X$. Hence the contact variety of $Q, P_2,\dots, P_k$ has the same properties
 of $\Gamma_k$. In particular ${\Pi_k}$ is tangent to $X$ along an irreducible  variety of the same dimension of $\Gamma_k$, which contains
 $P_2,\dots, P_k$ and $Q$. Since $Q\notin \Gamma_k$, this is excluded by the generality of $P_1,\dots,P_k$.
 \end{proof}

\begin{corollary} Assume that $\Gamma_k$ is irreducible (i.e. type I). Then the secant degree of $X$ is equal to the secant degree of 
the general contact locus $\Gamma_k$.
\end{corollary}

The situation becomes completely different when the general $\Gamma_k$ is of type II. 

\begin{remark}\label{remspaces} Assume that $\dim \sigma _{k+n}(X) =(k+n)(n+1)-1< r$,
so that $X$ is also weakly $(k+1)$-defective. Assume that a general {tangential} contact locus $\Gamma_{k+1}$ is formed by $k+1$
linearly independent linear spaces. Then also the general {tangential} contact locus $\Gamma_k$ is formed by $k$
linearly independent linear spaces.

Indeed, with no loss of generality, as in the proof of Lemma \ref{a2}, we may assume $\Gamma_k\subset\Gamma_{k+1}$.
Thus $\Gamma_k$ is of type II, and all its  components are contained in one component of $\Gamma_{k+1}$.
But clearly if one takes irreducible varieties $\Gamma^i$, $i=1,\dots,k$, of
the same dimension $t_k$, such that each $\Gamma^i$ spans a linear space ${\Pi^i}$ of dimension $>t_k$, and the linear
spaces ${\Pi^i}$'s are linearly independent, then the span of the union of the $\Gamma^i$'s cannot have dimension $kd+k-1$.
\end{remark} 

We point out the following example.

\begin{example} (\cite[Proposition 1.7]{bv}) Let $X$ be the Grassmanniann $G:=\mathbb{G}(\mathbb{P}^2, \mathbb{P}^7)$
and let $k=3$. Then $X$ is weakly $3$-defective and, for a general choice of $P_1,P_2,P_3\in X$, $\Gamma_3$ is the union 
of three disjoint linear spaces $L_1,L_2,L_3$ of dimension $3$, such that $L_i$ contains $P_i$. The span ${\Pi_3}$ of $\Gamma_3$
is also tangent to $X$ along a linear space $W=\mathbb{P}^5$, which misses the $P_i$'s and intersects each $L_i$ along a line.

One computes soon that the join of $L_i,L_j,W$ cannot fill ${\Pi_3}$, for obvious dimensional reasons.
It follows that the secant degree of $X$ equals the secant degree of ${\Gamma_3}$, which is $1$. Hence  the generic 
$3$-identifiability of $G$ holds. 
\end{example}

We can derive, as in the previous example a general statement for the case where the tangential contact variety 
$\Gamma_k$ is of type II.

\begin{proposition}\label{vipiace} Assume that the general $k$-tangential contact locus $\Gamma_k$ is of type $II$, i.e.
it is the union of  components $\Gamma^1, \ldots , \Gamma^k$. Assume that ${\Pi_k}$ is  tangent to $X$
along $\Gamma_k\cup\Gamma$, where $\Gamma\not\subset\Gamma_k$. If  $\langle \Gamma^1\cup \cdots  \cup \hat 
\Gamma^I\cup \cdots \cup \Gamma^k \cup \Gamma \rangle\neq {\Pi_k} $, where with {$\hat \Gamma^I$} we indicate any non empty subset of 
$\{\Gamma^1, \ldots , \Gamma^k\}$, then the $k$-th secant degree of $X$ is equal to the $k$-th secant degree of $\Gamma_k$.
\end{proposition}

\begin{corollary}\label{cipiace} In the assumptions of the previous proposition, assume the $\Gamma^i$'s are linear
subspaces of $\PP^r$, and linearly independent. Then $X$ is $j$-identifiable for any $j\leq k$.
\end{corollary}
\begin{proof} Clearly if $X$ is $k$-identifiable then it is also $j$-identifiable for every $j\leq k$, 
hence we need simply to show that $X$ is $k$-identifiable. This last follows since the union of linearly independent
linear spaces is identifiable. 
\end{proof}

On the other hand, there are cases (where $X$ is an {\it almost unbalanced} products of projective spaces) in which $\Gamma_k$
has type II and the previous proposition does not apply, moreover the $k$-th secant degree of $\Gamma_k$ is different from
the $k$-th secant degree of $X$.

A first example is the following.

\begin{example}\label{aumb} Consider the Segre variety $X$ of $\mathbb{P}^2\times \mathbb{P}^2\times \mathbb{P}^5$ and let $k=5$.
One computes that the general tangential contact locus  $\Gamma_5$ consist of $5$ linearly independent linear subspaces 
$\Gamma^1\cup\dots\cup\Gamma^5$, where each $\Gamma_i$ is a $\PP^4$ which passes through one of the points $P_i$.

The span ${\Pi_5}$ of $\Gamma_5$ is a $\mathbb{P}^{24}$ which is tangent to $X$ along an extra $\Gamma=\mathbb{P}^4$.
$\Gamma$ is linearly independent from any subset of $4$ among the $\Gamma^i$'s. Thus the $5$-th secant degree of $\Gamma_5$ is $1$,
but the secant degree of $X$ equals the secant degree of $\Gamma_5\cup\Gamma$, hence it is $6$.
\end{example}

A similar situation holds for the almost unbalanced cases defined in Section 8 of \cite{bco}.

\section{Applications to concrete examples}

We show below how Lemma \ref{a2} applies, thanks to the observations contained in Section \ref{secdeg}, to conclude
the identifiability of some
relevant cases of Segre products, Veronese varieties, Segre-Veronese varieties and moment varieties.
\medskip}

Non-trivial Segre products $X\subset\PP^r$ , {$X$ image of $\PP^{a_1}\times\cdots\times\PP^{a_s},$
 contain, for each $j=1,\dots,s$, linear spaces ${L_1,\dots, L_k}$, of dimension $a_j$ passing through $k$ general points, 
which are linearly independent, provided that $k(a_j+1)$ is smaller than $(r+1)$}. We do not know examples 
where these spaces are contact loci of weakly defective Segre products.
In any case, the previous lemma, together with the analysis of weakly defective varieties contained in \cite{cc2} and \cite{co},
yield the following.

\begin{theorem}\label{ip1}
Let $X\subset \PP^r$ be an integral and non-degenerate $n$-dimensional variety and let $k$ be a positive integer. 
Assume that $X$ is not uniruled by lines and that $\dim \sigma _{k+n-1}(X) ={(k+n-1)}(n+1)-1< r$ (in particular
 $X$ is not $(k+n-1)$-defective).
Then $X$ is not weakly $j$-defective (hence it is $j$-identifiable) for every $j\leq k$.
\end{theorem}
\begin{proof} If $X$ is not weakly $j$-defective, then the claim follows by \cite{cc2}. Assume that $X$ is weakly $j$-defective. 
Hence $X$ is also weakly $(j+i)$-defective for $i=1,\dots,n-1$. Since $X$ is not uniruled, then $\Gamma_{j+i}$ 
cannot be a union of linear spaces of  positive dimension, so by Lemma \ref{a2} we have:
$$1\leq \dim{\Gamma_j}<\dots <\dim(\Gamma_{j+n-1}).$$
It follows that $\dim(\Gamma_{j+n-1})\geq n$, a contradiction.
\end{proof}

Theorem \ref{ip1} is applicable to any embedding $X\subset \PP^r$ with $\Oo _X(1) 
\cong \Ll ^{\otimes 2}\otimes \Mm$ with $\Ll$ an ample line bundle and $\Mm$ a nef line bundle.
In particular it is applicable to any Veronese embedding of $\PP^n$, since in this case $\dim \sigma _i(X)$ 
is known for all $i$ by a famous theorem of Alexander and Hirschowitz (\cite{ah}). 
However, the result that we get from Theorem \ref{ip1} for symmetric tensors  is weaker than
the results of  \cite{gm, cov1}, where all the non-identifiable Veronese varieties are classified.
\medskip

Let us turn to the study of Segre varieties.
Fix integer $s\ge 2$, $n_i>0$, $1\leq i\le s$. Set $n:= n_1+\cdots +n_s$ and $r:=-1+\prod _{i=1}^{s}(n_i+1)$ 
and let $X\subset \PP^r$  be the  ($n$-dimensional) Segre embedding of the multiprojective space 
$\PP^{n_1}\times \cdots \times \PP^{n_s}$.  Since $X$ is uniruled by lines,  we cannot apply Theorem \ref{ip1} 
directly. On the other hand, the structure of linear spaces
contained in Segre embeddings of multiprojective spaces is well known.

\begin{remark}\label{b1} Let $\pi _i: X\to \PP^{n_i}$ denote the projection of the Segre variety
$X$ into its $i$-th factor. For any $i=1, \ldots ,s$ let $\Oo _X(\epsilon _i)$ be the line bundle $\pi _i^\ast (\Oo _{\PP^{n_i}}(1))$.
Each $\Oo _X(\epsilon _i)$ is a spanned line bundle and $\Oo _X(1) =\otimes _{i=1}^{s} \Oo _X(1)$. 
Let $L\subset X$ be any linear space of positive dimension. Since $\Oo _L(1)\cong
\oplus _{i=1}^{s} \Oo _X(\epsilon _i)_{|L}$, there is $j\in \{1,\dots ,s\}$ such that $\Oo _X(\epsilon _i)_{|L} 
\cong \Oo _L$ for all $i\ne j$, i.e. $L$ is contained in a fiber
of $\pi _j$. Fix an integer $k\ge 1$ such that $\dim \sigma _{k+1}(X) = (k+1)(n+1)-1 <r$ and assume that $X$ is weakly $(k+1)$-defective,
with general tangential $(k+1)$-contact locus $\Gamma _{k+1}$ equal to a union 
of linear spaces, say ${L_1,\dots ,L_k,L_{k+1}}$, of dimension $\epsilon _k >0$.
 Thus for each $i=1,\dots ,k+1$ there is a unique integer 
$j(i)\in \{1,\dots ,s\}$ such that $\pi _{j(i)}({L_i})$ is a point.
We get a function $\phi _{k+1}: \{1,\dots ,k+1\} \to \{1,\dots ,s\}$ defined by the formula $\phi _{k+1}(i): =j(i)$.
The function $\phi _{k+1}$ cannot depend on the choice of $k+1$ general points. Then, 
moving independently the points, by monodromy we get that this function $\phi _{k+1}$ must be invariant for the action
of the full symmetric group  on $\{1,\dots ,k+1\}$. 
Thus $j(i) =j(1)$ for all $i$.
\end{remark}

Using the previous remark, we can prove the following.

\begin{theorem}\label{ip3} Let $X$ be the Segre embedding of 
$\mathbb{P}^{n_1}\times \cdots \times \mathbb{P}^{n_k}$ in $\mathbb{P}^r$ as above.
Assume that for each $i\in \{1,\dots ,s\}$ there is  $j\in \{1,\dots ,s\}\setminus \{i\}$ 
such that $n_j=n_i$. Assume also that $\dim \sigma _{k+n-1}(X) ={(k+n-1)}(n+1)-1< r$.
Then $X$ is not weakly $j$-defective (hence it is $j$-identifiable) for every $j\leq k$.
\end{theorem}
\begin{proof}
Assume that $X$ is weakly $j$-defective. Thus it is also weakly $(j+i)$-defective, for $i=0,\dots,n-1$.

 Each general contact locus $\Gamma_{j+i}$ has dimension $t_{j+i}\leq n-1$ and so there cannot be a strictly increasing sequence 
 of integer $t_j,\dots,t_{j+n-1}$. By  Lemma \ref{a2} there is an integer $h\in \{j,\dots ,j+n-1\}$  
 such that $\Gamma _h$ is a union ${L_1\cup \cdots \cup L_h}$ of $h$ independent linear spaces
 of dimension $t_h$. Thus, by Remark  \ref{b1}, there is an integer $m\in\{1,\dots,s\}$
 such that, for all $i=1,\dots ,h$, $\pi _m({L_i})$ is a linear space of dimension $\dim {L_i}$, while
$\pi _x({L_i})$ is a point for all $x\in \{1,\dots ,s\}\setminus \{m\}$. 

Notice that the general contact locus $\Gamma _h$ is invariant for linear automorphisms of $\PP^r$ preserving $X$,
 in the sense that for any automorphism $\phi : \PP^r\to \PP^r$ with $\phi (X)=X$ the set $\phi (\Gamma _h)$
may be taken as a general contact locus. 
Thus $\pi _x(h({L_i}))$ is a point for each $x\in \{1,\dots ,s\}\setminus \{m\}$, 
while $\pi _m(h({L_i)})$ is a linear space of dimension $\dim \Pi_i$.
By assumption there is $z\in \{1,\dots ,s\}\setminus \{j\}$ with $n_z =n_j$. 
Hence there is an automorphism $\psi$ of $X$ exchanging the $m$-th factor and the $z$-th factor.
This automorphism extends to a linear automorphism $\phi : \PP^r\to \PP^r$ with $\phi (X)=X$, 
because the embedding $X\hookrightarrow \PP^r$ is induced by the complete linear system
$|\Oo _X(1,\dots ,1)|$ and $\psi ^\ast (\Oo _X(1,\dots ,1)) \cong \Oo _X(1,\dots ,1)$. 
We would have that $\pi _z({L_i})$ is a linear space of dimension $\dim {L_i}$, 
which yields a contradiction.
\end{proof}

The condition that for each $i\in \{1,\dots ,s\}$ there is  $j\in \{1,\dots ,s\}\setminus \{i\}$ 
such that $n_j=n_i$ restricts the range of application of Theorem \ref{ip3}. However, the theorem
applies to the products of many copies of a given $\PP^q$, i.e. to the case of cubic tensors (provided that one knows the
non-defectivity of the corresponding Segre embedding).

 In particular, in the case $n_i=1$ for all $i$, as an immediate corollary of \cite{cgg} and Theorem \ref{ip3},
  we obtain the following result.

\begin{corollary}\label{ip4}
Fix an integer $s\ge 5$ and an integer $k>0$ such that {$(k+s-1)(s+1) \leq 2^s -1$}. Let $X\subset \PP^{2^s-1}$ 
be the Segre embedding of $(\PP^1)^s$. Then $X$ is not weakly $k$-defective, hence it is $k$-identifiable.
\end{corollary}

Thus, we obtain that binary tensors, i.e. tensors of type $2\times \dots\times 2$ ($s$ times),
are $k$-identifiable for 
$$ {k \leq\frac{2^s-1}{s+1}-s+1.}$$
This bound is very near to $k_0=\lceil 2^s/(s+1)\rceil$ is the maximal integer for which $k$-identifiability
can hold.

As an immediate corollary of \cite[Theorem 3.1]{ah0} we get the following result.

\begin{corollary}\label{e1}
Take $X= (\PP^m)^s\subset \PP^r$, $m>1$, $r+1 = (m+1)^s$. Set $a:= \lfloor (m+1)^s/(ms+1)\rfloor$ and let $\epsilon$ be
the only integer such that $0 \leq \epsilon \leq m$ and $a \equiv
\epsilon \pmod{m+1}$. If either $\epsilon >0$ or $(k+ms)(sm+1) \leq (m+1)^s-1$, set $a':= a-\epsilon$. If $\epsilon =0$ and
$(ms+1)a = (m+1)^s$ set $a':= a-1$. For any positive integer
$k+ms-1\leq a'$ $X$ is not weakly $k$-defective and hence it is $k$-identifiable.
\end{corollary}
 
Take $X$ as in Corollary \ref{e1}. Note that if $k\ge a+2$, then $\sigma _{a+2}(X)$ has expected dimension $>r$ 
 and hence identifiability fails for $\sigma _{a+2}(X)$. In the special case $s=3$, in Corollary 3.10 it is sufficient to assume $3m(k+3m-1) < (m+1)^3$, 
because the secant varieties of $(\PP^m)^3$ are not defective, by \cite{li}. 

Complete results on the defectivity of Segre products with more than $3$ factors, except for the case of
many copies of $\PP^1$ and  for $(\mathbb{P}^m)^3$ (\cite{li}) and an almost complete for \cite{ah0}, are not available. 
As far as we know, the best procedure to find theoretically the
non defectivity of Segre products is the inductive method described in \cite{aop}. 
\smallskip

For Segre-Veronese varieties, which correspond to rank $1$ partially symmetric tensors, we can join Theorem
\ref{ip1} and Theorem \ref{ip3} to get:

\begin{theorem}\label{ip5} 
Fix integer $s\ge 2$, $n_i>0$, $d_i>0$, $1\leq i\le s$. Set $n:= n_1+\cdots +n_s$ and $ r = -1+\prod \binom {d_i+n_i}{d_i}.$
Let  $X\subset \PP^r$  be the  ($n$-dimensional) Segre-Veronese embedding of the multiprojective space 
$\PP^{n_1}\times \cdots \times \PP^{n_s}$, via the linear system $L_1^{d_1}\otimes\dots\otimes L_s^{d_s}$,
where each $L_i$ is the pull-back of $\Oo_{\PP^{n_i}}(1)$ in the projection onto the $i$-th factor. 

Assume that {$\dim \sigma _{k+n-1}(X) =(k+n-1)(n+1)-1< r$}. Assume also that either:
\begin{itemize}

\item $d_i>1$ for all $i$; or

\item for all $i$ with $d_i=1$ there exists $j\neq i$ such that $d_j=1$ and $n_i=n_j$.
\end{itemize}

Then $X$ is not weakly $j$-defective (hence it is $j$-identifiable) for every $j\leqq k$.
\end{theorem}
\begin{proof} Notice that $X$ is not uniruled by lines, unless $d_i=1$ for some $i$. When some $d_i$ is $1$, one can still
apply the trick introduced in the proof of Theorem \ref{ip3}.
\end{proof}

\begin{example}

C. Araujo, A. Massarenti and R. Rischter proved that
if $h\leq a^{\lfloor \log _2(d-1)\rfloor}$, then $\sigma _{h+1}(X)$ has the expected dimension (see \cite[Theorem 1.1]{amr} 
for a more precise result). 

For the case $s=2$ there is a very strong conjecture and several results supporting it (\cite{a, ab1, ab3}). 
When $s=2$ H. Abo and M. C. Brambilla  give a list of $9$ classes of secant varieties, which are
known to be defective and they conjecture that they are the only defective ones for $s=2$ (\cite[Conjecture 5.5]{ab3}. 
See \cite[Table 2]{a1} for the list of known (in 2010) non-defective cases when $s=2$.

H. Abo and M. C. Brambilla also gave many examples of defective secant varieties when $s=3,4$ (\cite{ab2}).
See \cite[\S 5]{ab3} for a general conjecture for the non-defectivity of Segre-Veronese varieties. 
If $d_i\ge 2$ for all $i$, then $X$ is not uniruled by lines, 
so we may apply  Theorem \ref{ip1} to obtain the non weak defectivity.
\end{example}

\begin{example}\label{r+1}C. Am\'endola, J.-C. Faug\`ere, K. Ranestad and B. Sturmfels (\cite{afs, ars}) studied the 
Gaussian moment variety $\Gg _{n,d} \subset \PP^N$, $N = \binom{n+d}{n}-1$, whose points are the vectors of all moments
of degree $\le d$ of an $n$-dimensional Gaussian distribution. We have $\dim \Gg _{n,d} =n(n+3)/2$. 
They proved that all the secant varieties of $\Gg _{1,d}$ are non-defective, but $\sigma _2(\Gg _{n,d})$ is defective
for all $n\ge 3$ (\cite[Theorem 13 and Remark 15]{ars}). To apply Theorem \ref{ip1} to $\Gg _{1,d}$ we need prove that 
$\Gg _{1,d}$ is not uniruled by lines. If so, Theorem \ref{ip1}  may be applied to $\sigma _k(\Gg _{1,d})$ for all $k$ with 
{$2(k+1)<d$}. We may assume $d\ge 4$. The variety $\Gg _{1,d}\subset \PP^d$ is described in \cite{afs} and \cite{ars}; 
it is an arithmetically Cohen-Macaulay surface of degree $\binom{d}{2}$ (\cite[Proposition 1 and the last part of its proof and 
Corollary 2]{afs}) and its singular locus is a line $L\subset \Gg _{1,d}$ (\cite[Lemma 4]{ars}). Let $\ell _L: \PP^d\setminus 
L\to \PP^{d-2}$ the note the linear projection of from $L$ and let $\pi : \Gg _{1,d}\setminus L \to \PP^{d-2}$ be the restriction 
of $\ell _L$. Assume that $\Gg _{1,d}$ is ruled by line and call $\Delta$ the one-dimensional family of lines giving the ruling. 
Since the Grassmannian are complete, $\Delta$ coves $\Gg _{1,d}$
and hence the closure of the image of $\pi$ would be an integral and non-degenerate curve $Y\subset \PP^{d-2}$ such that a 
general $L\in \Delta$ is mapped to a general point of $Y$. The surface $\Gg _{1,d}$ is covered by a family of degree $d$ 
rational normal curves of $\PP^d$, each of them tangent to $L$ at a different point. Hence $Y$ must be the rational 
normal curve of degree $d-2$ and $X$ would be contained in a degree $d-2$ $3$-dimensional cone $T\subset \PP^d$ 
with vertex $L$. Since $\deg (\Gg _{1,d}) =\binom{d}{2} > 3(d-2)$, Bezout theorem gives that every cubic hypersurface 
containing $\Gg _{1,d}$ contains $T$, contradicting the fact that $\Gg _{1,d}$ is cut out by cubics (\cite[Proposition 1]{afs}).
\end{example}

{
For Segre products, the most general result that we can obtain with our techniques is resumed in the following theorem. 
It can be applied to a wide class of Segre products, not covered by Theorem \ref{ip3}.

\begin{theorem} \label{ipx}
Let $X$ be the Segre embedding of a product $\mathbb{P}^{n_1}\times \cdots \times \mathbb{P}^{n_q}$ 
in $\mathbb{P}^r$, with $n=\sum n_i=\dim(X)$. Let $k$ be a positive integer. 
Assume that  {$\dim \sigma _{k+n-1}(X) =(k+n-1)(n+1)-1< r$} (in particular
 $X$ is not $(k+n-1)$-defective).
Then $X$ is $j$-identifiable for every $j\le k$.
\end{theorem}}
\begin{proof} Assume that $X$ is weakly $k$-defective. Then, as in the proof of Theorem \ref{ip1}, 
there exists some $j$, $k\leq j\leq k+n-2$, such that the general tangential contact loci $\Gamma_j$ and $\Gamma_{j+1}$ 
have the same dimension. Thus, by Lemma \ref{a2} and its proof, the contact loci $\Gamma_j$ and $\Gamma_{j+1}$ 
are formed by a union of linearly independent linear spaces $L_i$, say of dimension $s$. By Remark \ref{b1}, the $L_i$'s 
are contained in some component of the product, which, as in the proof of Theorem \ref{ip5}, is the same for all $i$'s. 
Say for simplicity that the $L_i$'s sits in the first component. Call $Y$ the Segre embedding of 
$\mathbb{P}^{n_2}\times \cdots \times \mathbb{P}^{n_q}$ and call $Q_i\in Y$ the point to which $L_i$ 
maps in the projection $X\to Y$. The $Q_i$'s are general points of $Y$.
 Since $X$ is not $j+1$ defective, the span $M$ of the $L_i$'s has dimension 
$(j+1)(s+1)-1$. If $M$ meets $X$ only in the union of the  $L_i$'s, which are linearly independent subspaces, then by
\cite[Example 2.2]{bc} and  Lemma \ref{gamma}
$X$ is $(j+1)$-identifiable. Otherwise there is a point $P\in (M\cap X)\setminus \bigcup L_i$. The point $P$ cannot belong to the same
fiber over $Y$ of some $L_i$, for otherwise $M$ contains the span of $L_i$ and $P$, and its dimension cannot be $(j+1)(s+1)-1$.
Thus the projection of $M\cap X$ to $Y$ contains a point $Q$ different from the $Q_i$'s. It follows that for a general choice of
$j$ points $Q_1,\dots,Q_j$ in $Y$ the span of the $Q_i$'s meets $Y$ in a point  $Q$ different from the $Q_i$'s.
By the trisecant lemma (\cite[p. 109]{acgh}, \cite[Example 1.8]{cc0}, \cite[Proposition 2.6]{cc1}), this means that the 
codimension of $Y$ in its span is at most $j-1$. Since the codimension of $Y$
is $\Pi (n_i+1)-\sum n_i -1$ (where the product and the sum range from $2$ to $q$), 
and moreover $(j+1)(n+1)< \Pi_{i=1}^q(a_1+1)$, after a short computation we get that:
$$ a_1 > \Pi_{i=2}^q (a_i+1)- \sum_{i=2}^q a_i,$$
i.e. $Y$ is {\it unbalanced} in the sense of \cite[Section 8]{bco}. Since we also have $j\geq a_1$, it follows by \cite[Theorem 4.4] {aop} 
that either $\sigma_{j+1}(X)$ covers $\PP^r$, or $X$ is $(j+1)$-defective. In both cases we have a contradiction.
\end{proof}

\smallskip

We should remark that the inductive method of \cite{aop} is extended to the study of weak defectivity in \cite{bco}. 
The results contained in \cite{bco} are, in several cases, improved by the results that one can get via Theorem \ref{ipx}.
\smallskip

From the computational point of view, let us point out the following:

\begin{remark}\label{comput} For specific examples of Segre products $X$, the {\it stacked hessian} method 
introduced in \cite{cov} can provide a positive answer to the $k$-identifiability of $X$. The stacked hessian method
requires the computations of the derivatives of the span of $k$ tangent spaces at general points,  when
the points move.

In order to prove the non $(k+n)$-defectivity of $X$, by the Terracini's lemma, one has just
to compute the dimension of the span of $k+n$ general tangent spaces (see e.g. \cite[Section 6]{cov2}). 
When $n$ is not too big with respect to $k$, this second method, together with
Theorem \ref{ipx}, can provide the $k$-identifiability of $X$ with a faster procedure.
\end{remark}

\section{Positive characteristic}\label{Sp}

In this section we work over an algebraically closed field $\KK$ of arbitrary characteristic.

\begin{definition}\label{defvs}
Let $X\subset \PP^r$ be an integral and non-degenerate variety. Set $n:= \dim (X)$. $X$ is said to be 
\emph{very strange} if for a general codimension $n$ linear subspace
$M\subset \PP^r$ the set $X\cap M$ is not in linearly general position in $M$, i.e. there is a hyperplane of $M$ 
containing at least $r-n+2$ points of $M\cap X$ (\cite{r}).
\end{definition}

We recall that $X\subset \PP^r$ is not very strange if a general $0$-dimensional section of $X$ is in uniform position 
(\cite[Ch. 2]{acgh},\cite[Introduction]{r}). Hence $X\subset \PP^r$ is not very
strange if either $\mathrm{char}(\KK )=0$ or $X$ is non-singular in codimension $1$ and $r\ge n+3$ 
(\cite[Corollaries 1.6 and 2.2 and Theorem 0.1]{r}). Thus none of the examples of the last section is very strange. 
$X$ is not very strange if a general curve section of $X$ is not strange (\cite[Lemma 1.1]{r}) and in particular
$X$ is not very strange if $\mathrm{char}(\KK )> \deg (X)$.

\begin{lemma}\label{a10}
Let $Y\subset \PP^r$ be an integral and non-degenerate variety, which is not very strange. Set $n:= \dim (Y)$. 
Fix a general $S\subset Y$ with $\sharp (S) \leq r-n-1$.
Then $S$ is the set-theoretic base locus of the linear system on $Y$ induced by $H^0(\PP^r,\Ii _S(1))$.
\end{lemma}
\begin{proof}
Let $N\subset \PP^r$ be a general linear space of dimension $n-r-1$. By Bertini's theorem the scheme $N\cap Y$ 
is a finite set of $\deg (Y)$ points. Since $Y$ is not very strange, the set $N\cap Y$ is in linearly general position in $N$. 
Hence for any $E\subset N\cap Y$ with $\sharp (E) =\sharp (S) \leq n-r-1$, the restriction
of $U:= H^0(\PP^r,\Ii _E(1))$ to $N\cap Y$ has $E$ as its set-theoretic base locus. Since $N$ is a linear space 
the restriction of $U$ to $Y$ has base locus contained in $N\cap Y$. Thus $E$ is the base locus of the restriction of $U$ to $T$. 
Since $Y$ is integral and non-degenerate, a general subset $A\subset Y$ with cardinality at most
$r-n$ spans a general subspace of $\PP^r$ with dimension $\sharp (A)-1$. 
Hence $E$ is a general subset of $Y$ with cardinality $\sharp (S)$.
\end{proof}

\begin{proof}[Proof of Theorem \ref{ii1}:] Fix a general $q\in \sigma _k(X)$. If $\dim \sigma _k(X) < kn+k-1$, 
then a dimensional count gives that $\Ss (q,X)$ is infinite. 

Hence to prove Theorem \ref{ii1} it is sufficient to find a contradiction to the existence of $A, B\in \Ss (q,X)$ with $A\ne B$. 
Set $E:= B\setminus A\cup B$ and $Z:= A\cup E$. We have $h^1(\PP^r,\Ii _Z(1)) >0$, 
because $q\in \langle A\rangle \cap \langle B\rangle$ and $q\notin \langle A\cap B\rangle$ (\cite[Lemma 1]{bb}), 
i.e. $Z$ does not impose $\sharp (Z)$ independent conditions to $W$.
We may also assume that $A$ and $B$ are general in their irreducible component of the constructible set $\Ss (q,X)$. 
Since $q$ is general, we may assume that $A$ (resp. $B$) is a general subset of $X$ with cardinality $s$, 
but of course $A\cup B$ is far from being general (in general).  We identify $X$ with $M$, so that $\Ll$ and $\Rr$ 
are line bundles on $X$  and $\Oo _X(1) \cong \Ll \otimes \Rr$. Since $A$ is general and $\sharp (A)\leq \dim V$, the set of all $f\in V$ 
vanishing at all points of $A$ is a linear subspace of $V$ with dimension $\dim (V) -s$. Hence $h^1(\PP^r,\Ii _A(1)) =0$. 
Since $A$ is general, by Lemma \ref{a10} there is a hypersurface
$T\in |V|$ with $A \subset T$ and $E\cap T = \emptyset$. Let $f\in V\setminus \{0\}$ be an equation of $T$. 
The multiplication by $f$ induces an isomorphism between $H^0(X,\Rr)$
and the linear subspace $W(-T):= H^0(X,\Ii _T\otimes \Ll \otimes \Rr )\cap W$ of all $g\in W$ vanishing on $T$. 
Moreover, for any finite subset $S\subset X\setminus T$
this isomorphism sends $H^0(X,\Ii _S\otimes \Rr )$ onto $W(-T-S) =\{h\in W: h_{|T\cup S}\equiv 0\}$. 
Since $\sharp (E)\leq \sharp (B)\leq h^0(\Rr)$ and $B$ is general in $X$, we have
$h^0(X,\Ii _E\otimes \Rr) = h^0(X,\Rr )-\sharp (E)$, i.e. $E$ impose $s$ independent conditions to $H^0(X,\Rr)$. 
Since $T\cap E =\emptyset$, we
get $\dim (W(-T-E)) = \dim (W(-T)) -\sharp (E)$. Hence $H^0(\PP^r,\Ii _Z(1)) = \dim W-\sharp (Z)$, 
i.e. $h^1(\PP^r,\Ii _Z(1))=0$, a contradiction.
\end{proof}

We give an example of application of Theorem \ref{ii1}.

\begin{example}\label{ii2}
Take $X= \PP^{n_1}\times \cdots \times \PP^{n_s}$ and positive integers $d_i>0$ and let $X\subset \PP^r$, 
$r= -1 +\prod _{i=1}^{s}\binom{n_i+d_i}{n_i}$. Write
$d_i =b_i+c_i$ with $b_i>0$, $c_i\ge 0$ and $c_j>0$ for some $j$. 
Set $n:= n_1+\cdots +n_s$. If we fix any  integer $k>0$ such that 
$$k\leq \min \{-2-n +\prod _{i=1}^{s}\binom{n_i+b_i}{n_i}, \prod _{i=1}^{s}  \binom{n_i+c_i}{n_i}\},$$
then Theorem \ref{ii1} applies and proves the identifiability of a general $q\in\PP^r$ of $X$-rank $k$.
\end{example}

\end{document}